\documentclass[11pt,a4paper]{article}

\usepackage[left=0.7cm, right=0.7cm, top=1.8cm, bottom=0.7cm]{geometry}

\usepackage{amsmath,amsthm,amsfonts,amssymb,amscd,cite,graphicx,epstopdf}

\usepackage[colorlinks=true, linkcolor=blue,citecolor=blue, urlcolor=blue]{hyperref}

\def\Dj{\hbox{D\kern-.73em\raise.30ex\hbox{-}
\raise-.30ex\hbox{}}}
\def\dj{\hbox{d\kern-.33em\raise.80ex\hbox{-}
\raise-.80ex\hbox{\kern-.40em}}}

\usepackage{titlesec}
\titleformat{\section}
{\normalfont\fontsize{12}{15}\bfseries}{\thesection}{1em.}{}

\titleformat{\subsection}
{\normalfont\fontsize{12}{15}\bfseries}{\thesubsection}{1em.}{}

\titleformat{\subsubsection}
{\normalfont\fontsize{12}{15}\bfseries}{\thesubsubsection}{1em.}{}

\newtheorem{corollary}{Corollary}[section]
\newtheorem{lemma}{Lemma}[section]

\newtheorem{theorem}{Theorem}[section]

\allowdisplaybreaks[4]

\let\oldbibliography\thebibliography
\renewcommand{\thebibliography}[1]{%
  \oldbibliography{#1}%
  \setlength{\itemsep}{-2pt}%
}

\begin{document}

\baselineskip=0.20in

\makebox[\textwidth]{%
\hglue-15pt
\begin{minipage}{0.6cm}	
\vskip9pt
\end{minipage} \vspace{-\parskip}
\begin{minipage}[t]{14.5cm}
\end{minipage}
\hfill
\begin{minipage}[t]{5.50cm}
\end{minipage}}
\vskip30pt


\begin{center}
{\large \bf \boldmath Harmonic-Arithmetic Index of (Molecular) Trees}\\[4mm]

\noindent
Abeer M. Albalahi$^1$, Akbar Ali$^{1,}\footnote{Corresponding author (akbarali.maths@gmail.com).}$, Abdulaziz M. Alanazi$^2$, Akhlaq A. Bhatti$^3$, Amjad E. Hamza$^1$\\[4mm]

\noindent
\footnotesize $^1${\it Department of Mathematics, Faculty of Science, University of Ha\!'il, Ha\!'il, Saudi Arabia}\\[2mm]
\noindent
$^2${\it Department of Mathematics, University of Tabuk, Tabuk, Saudi Arabia}\\[2mm]
\noindent
$^3${\it Department of Sciences and Humanities, National University of Computer and Emerging Sciences, Lahore, Pakistan}\\

\end{center}



\setcounter{page}{1}
\thispagestyle{empty}

\baselineskip=0.20in

\normalsize

 \begin{abstract}
 \noindent
Let $G$ be a graph. Denote by $d_x$, $E(G)$, and $D(G)$ the degree of a vertex $x$ in $G$, the set of edges of $G$, and the degree set of $G$, respectively. This paper proposes to investigate (both from mathematical and applications points of view) those graph invariants of the form $\sum_{uv\in E(G)}\varphi(d_v,d_w)$ in which $\varphi$ can be defined either using well-known means of $d_v$ and $d_w$ (for example: arithmetic, geometric, harmonic, quadratic, and cubic means) or by applying a basic arithmetic operation (addition, subtraction, multiplication, and division) on any of two such means, provided that $\varphi$ is a non-negative and symmetric function defined on the Cartesian square of $D(G)$. Many existing well-known graph invariants can be defined in this way; however, there are many exceptions too. One of such uninvestigated graph invariants is the harmonic-arithmetic (HA) index, which is
obtained from the aforementioned setting by taking $\varphi$ as the ratio of the harmonic and arithmetic means of $d_v$ and $d_w$.
A molecular tree is a tree whose maximum degree does not exceed four. Given the class of all (molecular) trees with a fixed order, graphs that have the largest or least value of the HA index are completely characterized in this paper.\\[2mm]
 \end{abstract}

\baselineskip=0.20in

\section{Introduction}

We refer the readers to the books \cite{Bondy08,Harary69,Chartrand16} for those graph-theory terms that we use in this article without defining them here.

For a graph $G$, its sets of edges and vertices are represented by $E(G)$ and $V(G)$, respectively. For a vertex $w\in V(G)$, define $N_G(w)=\{v\in V(G):~vw\in E(G)\}$. The degree $d_w$ of a vertex $w\in V(G)$ is defined as $d_w=|N_G(w)|$. A vertex of degree one is referred to as a pendent vertex. An edge incident to a pendent vertex is known as a pendent edge. The degree set $D(G)$ of $G$ is the set of all different elements of the degree sequence of $G$.

A function defined on the class of all graphs is said to be a graph invariant if its output is the same for all isomorphic graphs.
Graph invariants may be numerical quantities, polynomials, sets of numbers,  etc. Graph invariants that takes only numerical quantities are usually referred to as topological indices in chemical graph theory \cite{new2,new3}.

The topological indices that may be defined via the following formula \cite{Hollas,Gutman-13} are sometimes referred to as bond incident degree (BID) indices \cite{Vuki-CPL-10,Ali18}:
\begin{equation}\label{eq-BID-0009}
BID_\phi(G)=\sum_{vw\in E(G)} \phi(d_v,d_w),
\end{equation}
where $\phi$ is a non-negative function defined on the Cartesian square of the degree set $D(G)$ of $G$ satisfying the equation $\phi(d_v,d_w)=\phi(d_w,d_v)$. Many already introduced topological indices can be deduced from Equation \eqref{eq-BID-0009}; for example,
\begin{itemize}
  \item Equation \eqref{eq-BID-0009} yields the geometric-arithmetic index \cite{c26} when
  \[
  \phi(d_v,d_w)= \frac{2\sqrt{d_v\,d_w}}{d_v+d_w},
  \]
  \item   from Equation \eqref{eq-BID-0009}, one gets the arithmetic-geometric index (for example, see \cite{Vujo-21}\hspace{0.3mm}) by taking
  \[
  \phi(d_v,d_w)= \frac{d_v+d_w}{2\sqrt{d_v\,d_w}},
  \]
  \item the symmetric division deg (SDD) index \cite{Vuki-10} is deduced from Equation \eqref{eq-BID-0009} when
  \[
  \phi(d_v,d_w)= \frac{d_v}{d_w}+\frac{d_w}{d_v}.
  \]
\end{itemize}
This paper proposes to investigate (both from mathematical and applications points of view) those graph invariants of the form \eqref{eq-BID-0009} in which $\varphi$ can be defined either using well-known means of $d_v$ and $d_w$ (for example: arithmetic, geometric, harmonic, quadratic, and cubic means) or by applying a basic arithmetic operation (addition, subtraction, multiplication, and division) on any of two such means. Many existing well-known invariants can be defined in this way. For example:
\begin{itemize}
  \item If one takes $\phi(d_v,d_w)$ as the arithmetic mean of $d_v$ and $d_w$ then \eqref{eq-BID-0009} gives $M_1(G)/2$, where $M_1$ is the first Zagreb index (for example, see \cite{Gutman-13}\hspace{0.3mm}).
  \item If one takes $\phi(d_v,d_w)$ as the geometric mean of $d_v$ and $d_w$ then \eqref{eq-BID-0009} gives  $R_{1/2}(G)$, where $R_{1/2}$ is the reciprocal Randi\'c index \cite{Gutman-14}.
  \item If one takes $\phi(d_v,d_w)$ as the harmonic mean of $d_v$ and $d_w$ then \eqref{eq-BID-0009} gives  $2\cdot ISI(G)$, where $ISI$ is the inverse sum indeg index \cite{Vuki-10}.
  \item If one takes $\phi(d_v,d_w)$ as the quadratic mean of $d_v$ and $d_w$ then \eqref{eq-BID-0009} gives  $SO(G)/\sqrt{2}$, where $SO$ is the Sombor index  (for example, see \cite{Liu-22}\hspace{0.3mm}).
 \item If one takes $\phi(d_v,d_w)$ as the cubic mean of $d_v$ and $d_w$ then \eqref{eq-BID-0009} gives  $SO_{3}(G)/\sqrt{2}$, where $SO_{3}$ is a particular case of the $p$-Sombor index  \cite{Reti-21}.
  \item If one takes $\phi(d_v,d_w)$ as the ratio of the arithmetic and geometric means (or geometric and harmonic means) of $d_v$ and $d_w$ then \eqref{eq-BID-0009} gives  $AG(G)$.
  \item If one takes $\phi(d_v,d_w)$ as the ratio of the geometric and arithmetic means (or harmonic and geometric means) of $d_v$ and $d_w$ then \eqref{eq-BID-0009} gives  $GA(G)$.
\end{itemize}
Also, we remark here that the SDD index can be defined via the ratios of the arithmetic and harmonic means of end-vertex degrees of edges of $G$, that is
\[
SDD(G)=4\sum_{vw\in E(G)} \frac{(d_v+d_w)/2}{2d_v\,d_w/(d_v+d_w)} - 2\big| E(G)\big|.
\]
Motivated by the above-mentioned facts and from the recently introduced inverse symmetric division deg (ISDD) index \cite{Ghorbani-21}, we consider and study here the harmonic-arithmetic (HA) index, which is defined via the ratios of the harmonic and arithmetic means of end-vertex degrees of edges of $G$, that is
\[
HA(G)=\sum_{vw\in E(G)} \frac{2d_v\,d_w/(d_v+d_w)}{(d_v+d_w)/2}= \frac{4d_v\,d_w}{(d_v+d_w)^2}.
\]
A molecular tree is a tree whose maximum degree does not exceed four. Given the class of all\\
(i) trees,\\
(ii) molecular trees,\\
with a fixed order, graphs that have the largest or least value of the HA index are completely characterized in this paper.

\section{\boldmath Main results}
We start this section with the following elementary lemma.
\begin{lemma}\label{lem-0}
The function $\Phi$ defined by
\[
\Phi(x) = \frac{4x}{(x+1)^2}, \quad \text{where \ } x\ge 1,
\]
is strictly decreasing.
\end{lemma}

First, we determine graphs that have the largest or least values of the HA index from the class of all trees with a given order.

\begin{theorem}\label{lem-0.5}
For every fixed integer $n\ge4$, among all trees with $n$ vertices, the star graph $S_n$ and the path graph $P_n$ uniquely possess the smallest and largest values, respectively, of the HA index, which are equal to
\[
4\left(1-\dfrac{1}{n}\right)^2 \quad \text{and} \quad n-\left(\frac{11}{9}\right).
\]
\end{theorem}

\begin{proof}
Let $T$ be a tree with $n$ vertices, where $n\ge4$. Consider an arbitrary edge $vw\in E(T)$ with the assumption that $d_v\ge d_w$. By Lemma \ref{lem-0}, one has
\[
1\ge\frac{4d_vd_w}{(d_v+d_w)^2}=\dfrac{4\left(\frac{d_v}{d_w}\right)}{\left(\frac{d_v}{d_w}+1\right)^2}\ge \dfrac{4(n-1)}{n^2}
\]
where the equation
\[
\frac{4d_vd_w}{(d_v+d_w)^2}= \dfrac{4(n-1)}{n^2}
\]
holds if and only if $\frac{d_v}{d_w}=n-1$; that is, if and only if $(d_v,d_w)=(n-1,1)$, and the equation
\[
\frac{4d_vd_w}{(d_v+d_w)^2}= 1
\]
holds if and only if $\frac{d_v}{d_w}=1$; that is, if and only if $d_v=d_w$.
Thus,
\[
HA(T)=\sum_{vw\in E(T)}\frac{4d_vd_w}{(d_v+d_w)^2}\ge \sum_{vw\in E(T)}\dfrac{4(n-1)}{n^2}= 4\left(1-\dfrac{1}{n}\right)^2.
\]
where the equation
\[
HA(T)= 4\left(1-\dfrac{1}{n}\right)^2
\]
holds if and only if $T=S_n$. Also, if $E'(T)$ denotes the set of all pendent edges of $T$ then
\begin{align*}
HA(T)&=\sum_{vw\in E(T)\setminus E'(T)}\frac{4d_vd_w}{(d_v+d_w)^2}+\sum_{v'w'\in E'(T)}\frac{4d_{v'}d_{w'}}{(d_{v'}+d_{w'})^2}\\[2mm]
&\le \sum_{vw\in E(T)\setminus E'(T)}(1)+\sum_{v'w'\in E'(T)}\left(\frac{8}{9}\right)\\[2mm]
&=n-\left(\frac{1}{9}\right)\big|E'(T)\big|-1\\[2mm]
&\le n-\left(\frac{11}{9}\right),
\end{align*}
where the equation
\[
HA(T)= n-\left(\frac{11}{9}\right)
\]
holds if and only if $T=P_n$.
\end{proof}

From the proof of  Theorem \ref{lem-0.5}, the next result follows.

\begin{corollary}\label{lem-0.5ju-ds}
For every fixed integer $n\ge4$, among all connected graphs with $n$ vertices, only the regular graphs and the star graph $S_n$ attain the largest and least values of the HA index, respectively.
\end{corollary}

The next result also follows immediately from Theorem \ref{lem-0.5}.

\begin{corollary}\label{lem-0.5ju}
The path graph $P_n$ uniquely attains the maximum HA index in the class of all molecular trees with $n$ vertices for every $n$\, greater than $3$.
\end{corollary}

Moreover, for $n=4,5,$ the minimal version of Corollary \ref{lem-0.5ju} can be obtained from Theorem \ref{lem-0.5}, where the star graph is the extremal tree. In the rest of this paper, we focus our attention to the problem of determining graphs having the least value of the HA index from the class of all molecular trees of  fixed order $n\ge6$.\\
For a graph $G$, let \[n_t=\big|\{u\in V(G): d_u=t\}\big|\] and \[m_{s,t}=m_{t,s}=\big|\{uv\in E(G): (d_u,d_v)=(s,t)\}\big|.\]

\noindent
If $T$ is a non-trivial molecular tree of order $n$, then
\begin{equation}\label{MT-Eq-1}
HA(T)=\sum_{1\,\le\, s\, \le\, t\, \le\, 4}   \frac{4s\,t}{(s+t)^2}\,m_{s,t}\,,
\end{equation}
\begin{equation}\label{MT-Eq-2}
n_1 + n_2 + n_3 + n_4=n\,,
\end{equation}
\begin{equation}\label{MT-Eq-3}
n_1 + 2n_2 + 3n_3 + 4n_4=2(n-1)\,,
\end{equation}
\begin{equation}\label{MT-Eq-4}
\sum_{ \substack{ 1\,\leq\, s\,\leq\, 4 \\
         s\,\neq\, t}}m_{s,t}+2m_{t,t}=t \cdot n_{t}\quad \text{for $t=1,2,3,4$.}
\end{equation}
By solving \eqref{MT-Eq-2}--\eqref{MT-Eq-4} for $n_{1},n_{2},n_{3},n_4, m_{1,4},m_{4,4},$ and then plugging the values
of $m_{1,4}$ and $m_{4,4}$ into \eqref{MT-Eq-1}, one obtains
\begin{align}\label{MT-Eq-5}
HA(T)&= \frac{19 n-31}{25}
               +\left(\frac{83}{225}\right)m_{1,2}
          +\left(\frac{3}{20}\right)m_{1,3}
          +\left(\frac{6}{25}\right)m_{2,2}\nonumber\\[4mm]
&
         \quad \  +\left(\frac{3}{25}\right)m_{2,3}
          +\left(\frac{2}{225}\right)m_{2,4}
         +\left(\frac{2}{25}\right)m_{3,3}
         +\left(\frac{24}{1225}\right)m_{3,4}\,.
\end{align}
We take
\begin{align}\label{MT-Eq-6}
\Gamma_{\!\! HA}(T)&= \left(\frac{83}{225}\right)m_{1,2}
          +\left(\frac{3}{20}\right)m_{1,3}
          +\left(\frac{6}{25}\right)m_{2,2}+\left(\frac{3}{25}\right)m_{2,3} +\left(\frac{2}{225}\right)m_{2,4}
         +\left(\frac{2}{25}\right)m_{3,3}
         +\left(\frac{24}{1225}\right)m_{3,4}\nonumber\\[4mm]
&\approx  0.369m_{1,2} +0.150m_{1,3} +0.240m_{2,2} +0.120m_{2,3}
 +0.009m_{2,4} +0.080m_{3,3} +0.020m_{3,4}\,.
\end{align}
Then, \eqref{MT-Eq-5} yields
\begin{equation}\label{MT-Eq-7}
HA(T)= \frac{19 n-31}{25} + \Gamma_{\!\! HA}(T)\,.
\end{equation}

\begin{lemma}\label{MT-Lem-1}
For a molecular tree $T$, if either
\begin{description}
  \item[(i).] $\max\{m_{1,2}, m_{1,3},m_{2,2}, m_{3,3},m_{2,3}\}\ge1$, or
  \item[(ii).] $n_3\ge1$ and $n_2\ge3$,
\end{description}
then
$$
\Gamma_{\!\! HA}(T) > \frac{8}{225}~~
(\approx 0.036)\,.
$$

\end{lemma}

\begin{proof}
We note that if any of the integers $m_{1,2},m_{2,2},m_{1,3},m_{3,3}, m_{2,3},$ is non-zero,
then the desired conclusion holds by \eqref{MT-Eq-6}.
Assume that $m_{1,2}=m_{1,3}=m_{2,2}= m_{2,3}=m_{3,3}=0$, $n_3\ge1$, and $n_2\ge3$.
By using \eqref{MT-Eq-4}, one has $m_{2,4}=2n_2$ and $m_{3,4}=3n_3$. Now, from \eqref{MT-Eq-6} it follows that
\begin{align}\label{MT-Eq-9}
\Gamma_{\!\! HA}(T)&= \left(\frac{4}{225}\right)n_{2}
          +\left(\frac{72}{1225}\right)n_{3}> \frac{8}{225}\,,\nonumber
\end{align}
as desired.
\end{proof}

\begin{theorem}
For $n\ge6$, among all molecular trees of order\, $n$,
\begin{description}
  \item[(i).] the trees containing no vertex of degree $3$ and containing exactly one vertex of degree $2$, which is adjacent to two vertices of degree $4$, are the only trees with the minimum HA index and that minimum value is equal to
      \[
      \frac{19 n-31}{25}+\frac{4}{225}\,,
      \]
      whenever $n\equiv0\pmod{3}$ with $n\ne6;$
  \item[(ii).] the trees containing no vertex of degree $3$ and containing exactly two vertices of degree $2$, each of which is adjacent to two vertices of degree $4$, are the only trees with the minimum HA index and that minimum value is
      \[
      \frac{19 n-31}{25}+\frac{8}{225}\,,
      \]
      whenever $n\equiv1\pmod{3}$ with $n\not\in\{7,10\};$
  \item[(iii).] the trees containing neither any vertex of degree $2$ nor any vertex of degree $3$ are the only trees with the minimum HA index and that minimum value is
      \[
      \frac{19 n-31}{25}\,,
      \]
      whenever $n\equiv2\pmod{3}$.
\end{description}
If $n=6,7,10,$ then the unique tree with the minimum HA index among all molecular trees of order $n$\, is depicted in Part (a), Part (b), Part (c), respectively, of Figure \ref{Fig-1}.

\end{theorem}

\begin{figure}[!ht]
 \centering
  \includegraphics[width=0.6\textwidth]{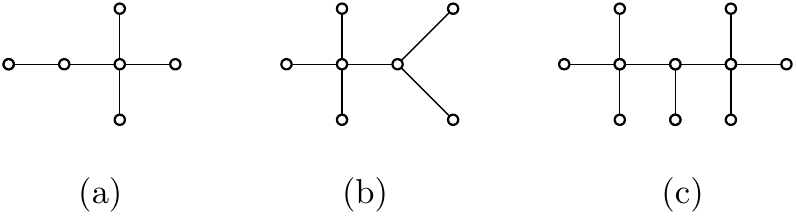}
   \caption{The trees with the minimum $HA$ value among all molecular trees of order (a) $n=6$, (b) $n=7$, (c) $n=10$.}
    \label{Fig-1}
     \end{figure}

\begin{proof}
First, we assume that $n\ge8$ and $n\ne10$. Let $T$ be a molecular tree of order $n$.\\[2mm]
{\bf Case 1.} Either $\max\{m_{1,2}, m_{1,3},m_{2,2},m_{3,3},m_{2,3}\}\ge1$ or $n_2\ge3$ and  $n_3\ge1$.\\
In this case, from Lemma \ref{MT-Lem-1} and Equation \eqref{MT-Eq-7}, it follows that
\begin{align*}
HA(T)&> \frac{19 n-31}{25} + \frac{8}{225}\\[4mm]
&> \frac{19 n-31}{25} + \frac{4}{225}\\[4mm]
&> \frac{19 n-31}{25},
\end{align*}
as desired.\\[2mm]
{\bf Case 2.} It holds that $m_{1,2}= m_{2,2}=n_3=0$ and $n_2\le2$.\\
By using \eqref{MT-Eq-2} and  \eqref{MT-Eq-3}, one has $n_2\equiv n-2\pmod{3}$,
which yields
$$
n_2=
\begin{cases}
1 & \text{when $n\equiv0\pmod{3}$,}\\[2mm]
2 & \text{when $n\equiv1\pmod{3}$,}\\[2mm]
0 & \text{when $n\equiv2\pmod{3}$.}
\end{cases}
$$
Thus, from \eqref{MT-Eq-4}, one has
\[m_{2,4}=
\begin{cases}
2 & \text{when $n\equiv0\pmod{3}$,}\\[2mm]
4 & \text{when $n\equiv1\pmod{3}$,}\\[2mm]
0 & \text{when $n\equiv2\pmod{3}$.}
\end{cases}
\]
The desired conclusion now follows \eqref{MT-Eq-5}.

\begin{figure}[!ht]
 \centering
  \includegraphics[width=0.95\textwidth]{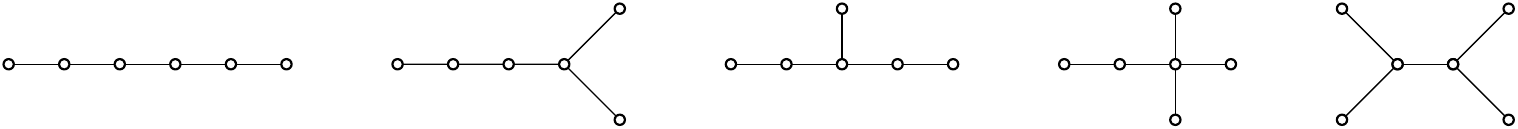}
   \caption{All the molecular trees of order $6$.}
    \label{Fig-2}
     \end{figure}

It remains to prove the result for $n=6,7,10$. Note that, by Corollary \ref{lem-0.5ju}, the inequality $HA(T)<HA(P_n)$ holds for every molecular tree $T$, different from the path graph $P_n$, of order $n\ge4$.

All the molecular trees of order $6$ are depicted in Figure \ref{Fig-2}; among all of them, the tree shown in Figure \ref{Fig-1}(a) has the minimum $\Gamma_{\!\! HA}$ value because for every other tree different from the path graph, the inequality $m_{1,2}+m_{1,3}\ge3$ holds and thence from \eqref{MT-Eq-5} and \eqref{MT-Eq-6} the desired conclusion follows for $n=6$.

\begin{figure}[!ht]
 \centering
  \includegraphics[width=0.95\textwidth]{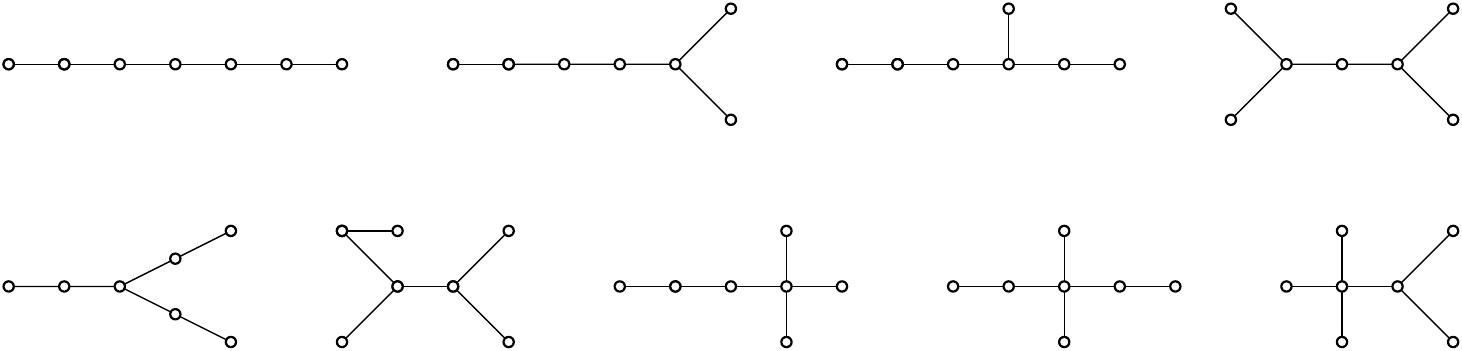}
   \caption{All the molecular trees of order $7$.}
    \label{Fig-3}
     \end{figure}

All the molecular trees of order $7$ are depicted in Figure \ref{Fig-3}; among all of them, the tree shown in Figure \ref{Fig-1}(b) has the minimum $\Gamma_{\!\! HA}$ value because for every other tree, at least one of the inequalities $\min\{m_{1,2},m_{2,2}\}\ge1$, $m_{1,2}\ge2$, $m_{1,3}\ge3$, holds and thence from \eqref{MT-Eq-5} and \eqref{MT-Eq-6} the desired conclusion follows for $n=7$.

Finally, in what follows, we assume that $n=10$. We claim that for any tree $T$ of order $10$ different from the one depicted in Figure \ref{Fig-1}(c), the following inequality
\begin{equation}\label{eq-09jjjhy9}
\Gamma_{\!\! HA}(T)> \frac{927}{4900}~~(\approx 0.189)
\end{equation}
holds.
If any of the three inequalities $\max\{m_{1,2},m_{2,2}\}\ge1$, $\max\{m_{2,3},m_{1,3}\}\ge2$, $m_{3,3}\ge3$, holds then from \eqref{MT-Eq-6}, the inequality \eqref{eq-09jjjhy9} follows. Thus, in the following, assume that $m_{1,2}=m_{2,2}=0$, $\max\{m_{2,3},m_{1,3}\}\le1$, and $m_{3,3}\le2$.

If $m_{2,3}=1$. Let $uv\in E(T)$ be the edge satisfying $(d_u,d_v)=(2,3)$. Then $N_T(u)=\{v,w\}$, with $w$ being a neighbor of degree $4$, because $m_{1,2}=m_{2,2}=0$. Consequently, none of the neighbors of $v$ has degree $4$ because $n=10$, and hence \eqref{eq-09jjjhy9} follows from  \eqref{MT-Eq-6}. Thence, in what follows, assume that $m_{1,2}=m_{2,2}=m_{2,3}=0$, $m_{1,3}\le1$, and $m_{3,3}\le2$. From $m_{1,2}=m_{2,2}=m_{2,3}=0$, it follows that if $T$ contains any vertex of degree $2$ then each of its neighbors has degree $4$; but $n=10$, which implies that $m_{2,4}=n_2=0$.
In the remaining proof, it is assumed that $n_2=0$, $m_{1,3}\le1$, and $m_{3,3}\le2$.

If $m_{3,3}\ne0$, let $x\in V(T)$ be a vertex of degree $3$ having at least one neighbor of degree $3$. Since $n=10$, $m_{3,3}\le2$, and $n_2=0$, the set $N_T(x)$ contains at least one vertex of degree $1$. Then \eqref{eq-09jjjhy9} holds by \eqref{MT-Eq-6}.

Finally, we assume that $n_2=m_{3,3}=0$ and $m_{1,3}\le1$. Since $n=10$, it holds that $n_3\ge1$ and $m_{1,3}=1$, which implies that $m_{3,4}=2$ and hence $T$ is the tree depicted in Figure \ref{Fig-1}(c). Therefore, $\Gamma_{\!\! HA}(T)= \frac{927}{4900}$.

Now, the desired result for $n=10$ follows from \eqref{MT-Eq-7}.
\end{proof}

\section{Concluding remarks}
A new graph invariant, namely the harmonic-arithmetic (HA) index, has been proposed and studied in this paper. One of the motivations for introducing the HA index came from the fact that many existing BID indices can be defined using well-known means; for example, arithmetic, geometric, harmonic, quadratic, and cubic means. Given the class of all (molecular) trees with a fixed order, graphs that have the largest or least value of the HA index have completely been characterized in this paper.

In \cite{Ghorbani-21}, the absolute values of the correlation coefficient between the ISSD index and
(i) entropy, (ii) acentric factor, for octane isomers were reported as (i) 0.87, (ii) 0.89, respectively. As one of the motivations behind proposing and studying the HA index is the ISDD index, we compute the absolute values of the correlation coefficient between the HA index and
(i) entropy, (ii) acentric factor, for octane isomers as (i) 0.91, (ii) 0.92, respectively. This indicates that the HA index performs slightly better than the ISDD index in predicting the aforementioned properties of octane isomers.
The experimental values for the mentioned properties of octane isomers are available at \\ [1mm] \url{https://web.archive.org/web/20180912171255if\_/http://www.moleculardescriptors.eu/index.htm}\\[1mm]
\indent
If one takes $\phi(d_v,d_w)$ as the ratio between quadratic mean and geometric mean then \eqref{eq-BID-0009} gives
\[\sum_{vw\in E(G)}\sqrt{\frac{1}{2}\left(\frac{d_v}{d_w}+\frac{d_w}{d_v}\right)},\]
which is certainly another variant of the SDD index and hence can be referred to as the ``modified symmetric division deg index''. It seems to be interesting also to investigate the chemical applicability and mathematical aspects of the modified symmetric division deg index.

\section*{Acknowledgement}

This work is partially supported by Scientific Research Deanship, University of Ha\!'il, Saudi Arabia, through project numbers RG-22\,002 and RG-22\,005.


\end{document}